\documentclass{elsarticle}
\usepackage{fullpage}
\usepackage{amsmath}
\usepackage{amsthm}
\usepackage{amsfonts}
\usepackage{amssymb}
\usepackage{float}
\usepackage{amscd}
\usepackage{sidecap}
\newtheorem{theorem}{Theorem}[section]
\newtheorem{lemma}[theorem]{Lemma}

\newtheorem{mydef}[theorem]{Definition}
\newtheorem{remark}[theorem]{Remark}
\newtheorem{hyp}[theorem]{Hypothesis}
\theoremstyle{remark}
\numberwithin{equation}{section}
\begin{document}

\begin{frontmatter}
\title{The Instability of the Hocking-Stewartson Pulse and its Geometric Phase in the Hopf Bundle}
\author[unc]{Colin Grudzien\corref{cor1}} 
\ead{cgrudz@email.unc.edu}
\address[unc]{Department of Mathematics University of North Carolina at Chapel Hill, Phillips Hall, CB3250 UNC-CH, Chapel Hill, NC 27599-3250}

\cortext[cor1]{Corresponding author}
\begin{abstract}
This work demonstrates an innovative numerical method for counting and locating eigenvalues with the Evans function.  Utilizing the geometric phase in the Hopf bundle, the technique calculates the winding of the Evans function about a contour in the spectral plane, describing the eigenvalues enclosed by the contour for the Hocking-Stewartson pulse of the complex Ginzburg-Landau equation.  Locating eigenvalues with the geometric phase in the Hopf bundle was proposed by Way \cite{WAY2009}, and proven by Grudzien, Bridges \& Jones \cite{Grudz2015}.  Way demonstrated his proposed method for the Hocking-Stewartson pulse, and this manuscript redevelops this example as in the proof of the method in \cite{Grudz2015}, modifying his numerical shooting argument, and introduces new numerical results concerning the phase transition. 
\end{abstract}
\end{frontmatter}

\section{Introduction}
\label{section:intro}
Way in his PhD thesis \cite{WAY2009}, supervised by Bridges at the University of Surrey, developed numerical results which supported the hypothesis that parallel transport in the Hopf bundle could locate and measure the multiplicity of of eigenvalues for linearizations of reaction-diffusion, differential operators.  The eigenvalue problem for the operator linearized about the steady state gives rise to a dynamical system on $\mathbb{C}^n$.  The Hopf bundle is represented as $S^{2n-1}\subset \mathbb{C}^n$ over the base space $\mathbb{C}P^{n-1}$, and thus has a realization in the phase space for an arbitrary dynamical system.  By projecting $\lambda$ dependent special solutions in $\mathbb{C}^n$ onto $S^{2n-1}$ the dynamics induce parallel transport in the Hopf bundle. In the fiber $S^1$ the parallel transport gives a winding number that, relative to a reference phase for the contour, counts the multiplicity of eigenvalues enclosed by the $\lambda$ parameter, as proven by Grudzien, Bridges \& Jones \cite{Grudz2015}.  Way demonstrated his method for the Hocking-Stewartson pulse of the complex Ginzburg-Landau equation, utilizing the exterior algebra formulation as Afendikov \& Bridges did for the Evans function \cite{BRI01}.  This work re-examines this example and the method---\S \ref{section:detbundle} will re-develop this example as in the proof of the method by Grudzien, Bridges \& Jones \cite{Grudz2015}, and with the example framed in this context, \S \ref{section:numerics} introduces new numerical results concerning the phase transition.  The numerical example highlights two important features that differentiates the method of geometric phase from other Evans function formulations; namely, the geometric phase method it utilizes \textit{either} the stable or unstable manifold for the computation of the eigenvalues, and the computation of the eigenvalue is continuous in the wave parameter, yielding a continuous accumulation of the eigenvalue driven by the system dynamics.
\section{The Hocking-Stewartson Pulse of the Complex Ginzburg-Landau Equation}
\label{section:detbundle}

The scaled, complex Ginzburg-Landau equation is given by
\begin{align}\label{eq:cgl}
\rho e^{i \psi} Y_t =Y_{xx} - (1+i \omega)^2 Y +(1+i \omega)(2+i \omega)\mid Y \mid^2 Y
\end{align} 
where $\rho>0$, $\psi$ and $\omega$ are specified real parameters for the system.  
The Hocking-Stewartson pulse is the steady state solution for the Complex Ginzburg-Landau equation, given by
\begin{align}\label{eq:HSpulse}
Y(x,t) = \left(cosh(x)\right)^{-1-i\omega}
\end{align}
Bates \& Jones \cite{Bates89} prove that stability of a steady state can demonstrated by stability of the linearization of equation (\ref{eq:cgl}), therefore, consider the linearization about the pulse.

\begin{mydef}\label{mydef:L}
Let $\mathcal{L}$ be the linearization of the operator defining equation (\ref{eq:cgl}) about the Hocking-Stewartson pulse (\ref{eq:HSpulse}).  
\end{mydef}

Considering solutions proportional to $e^{\lambda t}$, one can derive a non-autonomous system on $\mathbb{C}^4$ with asymptotic limits in $x$, as done by Afendikov \& Bridges \cite{BRI01}.  The system will be of the form
\begin{equation}\label{eq:C4} \begin{matrix}
v' = A(x,\lambda) v & v \in \mathbb{C}^4 \\ \\
 \lim_{x \rightarrow \pm \infty} A(x,\lambda) = A_{\pm\infty}(\lambda) & ' = \frac{d}{dx}
\end{matrix}\end{equation}
and it is equivalent to the eigenvalue problem in the following sense.  
\begin{lemma} 
A solution to the $\lambda$ dependent system on $\mathbb{C}^4$ is continuous and bounded if and only if it is an eigenfunction and $\lambda$ is an associated eigenvalue.
\end{lemma}
\begin{proof}
This is proven by Alexander, Gardner \& Jones \cite{AGJ1990} and the reader is referred there for the proof.
\end{proof}
\begin{remark}
The linearization of the complex Ginzburg-Landau equation about the pulse, $\mathcal{L}$ in definition \ref{mydef:L}, has essential spectrum on the set
\begin{align}
S_{ess} =\{\rho^{-1} e^{\mp i \psi} (\omega - s^2 -1)\mp 2 i \rho^{-1}\omega e^{\mp i \psi}, s \in \mathbb{R}^+\}
\end{align}
and for the parameter values $\omega=3$, $\rho = \frac{1}{\sqrt{5}}$, and $\psi = \arctan(2)$ there is a known double eigenvalue at $\lambda=0$, and simple eigenvalues at approximately $\lambda=-6.6357$ and $\lambda=15$ estimated by Afendikov \& Bridges \cite{BRI01}.
\end{remark}  
\begin{mydef}\label{mydef:splitting}
Assume $\Omega\subset \mathbb{C}$ is open, simply connected and contains only discrete eigenvalues of $\mathcal{L}$. System (\ref{eq:C4}) is said to \textbf{split} in $\Omega$ if $A_ {\pm \infty}(\lambda)$ are hyperbolic and each have exactly $k$ eigenvalues of positive real part (\textbf{unstable eigenvalues}) and $n-k$ eigenvalues of negative real part (\textbf{stable eigenvalues}), including multiplicity, for every $\lambda \in \Omega$. 
\end{mydef}
\begin{lemma}
For the parameter values $\omega=3$, $\rho=\frac{1}{\sqrt{5}}$, and $\psi = \arctan(2)$ system (\ref{eq:C4}) splits on the domain $\{\lambda\in\mathbb{C} : Re(\lambda)>0\}$.  Moreover, $\Omega \subset \mathbb{C}$ can be chosen such that $\{\lambda\in\mathbb{C} : Re(\lambda)>0\} \subset \Omega$ and $-6.6357 \in \Omega$.
\end{lemma}
\begin{proof}
Afendikov \& Bridges \cite{BRI01} demonstrate that the autonomous limits, $A_{\pm\infty}(\lambda)$, each have exactly 2 stable and unstable eigenvalues respectively, for each $\lambda$ such that $Re(\lambda)>0$, and in general for $\lambda \notin S_{ess}$.  For $\omega=3$, $\rho=\frac{1}{\sqrt{5}}$, and $\psi = \arctan(2)$, the essential spectrum is a curve in $\mathbb{C}$ that does not intersect $-6.6357 $; therefore an open $\Omega \subset \mathbb{C}$ can be chosen containing $\lambda=-6.6357$ without intersecting the essential spectrum, and for such an $\Omega$, system (\ref{eq:C4}) splits on the domain.  
\end{proof}
\begin{remark}
With the splitting condition satisfied, one may construct the Evans function as done by Alexander, Gardner \& Jones \cite{AGJ1990}---the method of geometric phase was proven by Grudzien, Bridges \& Jones \cite{Grudz2015} under these hypotheses for the Evans function. 
\end{remark}

In order to capture the winding of the unstable manifold of the asymptotic system $A_{-\infty}(\lambda)$, the compound matrix method is utilized, defining a dynamical system on the exterior algebra $\Lambda^2(\mathbb{C}^4)$.

\begin{mydef}
Let the matrix $A(\lambda,x)$ define dynamical system of the form (\ref{eq:C4}) on $\mathbb{C}^4$.  If $\{z_j\}^2_1$ are vectors in $\mathbb{C}^4$, the corresponding $A^{(2)}$ system on $\Lambda^2(\mathbb{C}^4)\equiv \mathbb{C}^{4\choose 2} \equiv \mathbb{C}^6$ is defined:
\begin{align}\label{eq:A^ksystem}
Z'=A^{(2)}(\lambda, x) Z =& \frac{d}{dx}(z_1 \wedge z_2) \\ 
                         = & A(\lambda,x)z_1 \wedge z_2 + z_1\wedge A(\lambda,x) z_2\\ 
 A^{(2)}_{\pm\infty}(\lambda) =& \lim_{x \rightarrow \pm\infty} A^{(2)}(\lambda,x)
\end{align}
\end{mydef}
\begin{remark}
The system (\ref{eq:A^ksystem}) yields coordinates for the evolution of the two dimensional subspaces of $\mathbb{C}^4$, which allows one to consider the evolution of the unstable manifold for the asymptotic system $A_{-\infty}(\lambda)$, and particularly the winding it accumulates both in the sense of the relative geometric phase considered by Grudzien, Bridges \& Jones \cite{Grudz2015} and the Chern number of the determinant bundle of the unstable manifold, constructed by Alexander, Gardner \& Jones \cite{AGJ1990}.  By way of the proof in \cite{Grudz2015}, these winding formulations are seen to be equivalent for special solutions which allow trivializations of the determinant bundle.  
\end{remark}
Explicitly, Afendikov \& Bridges derive the compound matrix system 
\begin{equation}\label{eq:A^2system} \begin{matrix}
u_x = A(\lambda, x)u & x\in \mathbb{R} & \lambda \in \mathbb{C} & u\in \mathbb{C}^6 \\ \\

A(\lambda) &=& 
\begin{pmatrix} 
0 & 0 & 1 & -1 & 0 & 0 \\
a_{32} & 0 & 0 & 0 & 0 & 0 \\
a_{42} & 0 & 0 & 0 & 0 & 1 \\
-a_{31} & 0 & 0 & 0 & 0 & -1 \\
-a_{41} & 0 & 0 & 0 & 0 & 0 \\
0 & -a_{41} & a_{31} & -a_{42} & a_{32} & 0 
\end{pmatrix}
\\ \\
\end{matrix} \end{equation}
with components defined
\begin{align*}
a_{31}=&\lambda \rho \cos(\psi) + 1- \omega^2 - (2 - \omega^2)(\hat{q}_2^2 +3 \hat{q}_1^2) + 6 \omega \hat{q}_1 \hat{q}_2 \\
a_{32}=& - \lambda \rho \sin(\psi) -2 \omega - 2(2- \omega)\hat{q}_1\hat{q}_2 + 3 \omega (\hat{q}_1^2 +3 \hat{q}^2_2) \\
a_{41}=&\lambda \rho \sin(\psi) + 2 \omega -2(2- \omega)\hat{q}_1\hat{q}_2 -3\omega(3\hat{q}_1^2 - \hat{q}_2^2)\\
a_{42} =&\lambda \rho cos(\psi) +1- \omega^2 - (2 - \omega^2)(\hat{q}_1^2 +3 \hat{q}_2^2) - 6 \omega \hat{q}_1 \hat{q}_2
\end{align*}
and $\hat{q}_1,\hat{q}_2$ derived from the expression for the pulse in $\mathbb{C}^4$, where
\begin{align*}
\hat{q}_1 =& \frac{\cos\left(\omega\log(\cosh(x))\right)}{\cosh(x)} \\
\hat{q}_2 =& \frac{-\sin\left(\omega\log(\cosh(x))\right)}{\cosh(x)}
\end{align*}
The non-autonomous system (\ref{eq:A^2system}) has the symmetric asymptotic limits
\begin{equation}\label{eq:A^2infty} \begin{matrix}
A_\infty(\lambda) = \text{lim}_{x \rightarrow \pm \infty} A(\lambda,x) &=&
\begin{pmatrix}
0 & 0 & 1 & -1 & 0 & 0 \\
-p(\lambda) & 0 & 0 & 0 & 0 & 0 \\
\eta(\lambda) & 0 & 0 & 0 & 0 & 1 \\
-\eta(\lambda) & 0 & 0 & 0 & 0 & -1 \\
-p(\lambda) & 0 & 0 & 0 & 0 & 0 \\
0 & -p(\lambda) & \eta(\lambda) & -\eta(\lambda) & -p(\lambda) & 0 
\end{pmatrix}
\end{matrix} \end{equation}
where the parameters are defined
\begin{align}
p(\lambda) & = 2\omega + \lambda\rho\sin(\psi) \\
\eta(\lambda) &= 1 - \omega^2 + \lambda \rho \cos(\psi)
\end{align}

The method of geometric phase makes use of a loop of eigenvectors, chosen analytically in the parameter $\lambda$, that correspond to the dominant unstable eigenvalue for the limiting systems at $\pm \infty$.  The method is to measure the relative winding the loop at $-\infty$ accumulates through its evolution as it meets a corresponding loop of eigenvectors at $+\infty$.  An algorithm for constructing these loops of eigenvectors is given by Humpherys, Standstede \& Zumbrun \cite{hump06}.  However, for the asymptotic system (\ref{eq:A^2infty}) such bases can be constructed explicitly.  The unique eigenvalues of most positive and most negative real part for system (\ref{eq:A^2infty}) are given by $\sigma^+,\sigma^-$ respectively, and have associated eigenvectors
\begin{align}\label{eq:ueigenvector}
&\begin{Bmatrix} 
\sigma^+ = \sqrt{2}\sqrt{\eta + \sqrt{\eta^2 + p^2}} ,&  \hspace{2mm}& X^+(\lambda) =
\begin{pmatrix} 2 \sigma^+ \\ -2p \\ (\sigma^+)^2\\ -(\sigma^+)^2 \\ -2p \\ \sigma^+\big((\sigma^+)^2 -2 \eta\big) \end{pmatrix} 
\end{Bmatrix}\\
&\begin{Bmatrix}
\sigma^- = -\sqrt{2}\sqrt{\eta + \sqrt{\eta^2 + p^2}} ,& & X^-(\lambda) =
\begin{pmatrix} 2 \sigma^- \\ -2p \\ (\sigma^-)^2 \\ -(\sigma^-)^2
\\-2p \\ \sigma^-\big((\sigma^-)^2 -2 \eta\big) \end{pmatrix}\end{Bmatrix}.
\end{align}
These eigenvectors correspond to the Grassmann coordinates for the un/stable subspace of the asymptotic system on $\mathbb{C}^4$.  
\begin{remark}
For $x<0$ and $\mid x\mid$ taken sufficiently large, the $\lambda$ dependent initial conditions defined by (\ref{eq:ueigenvector}) approximate the unstable manifold of the fixed point $0$ for the asymptotic system, and can be used for initial conditions for the numerical method.
\end{remark}
\section{The Method of Geometric Phase}
The strategy for determining the eigenvalues of the linear operator $\mathcal{L}$ with the Evans function is to choose a contour $K$ in the complex plane which does not intersect the spectrum of $\mathcal{L}$ and to calculate the winding of the Evans function around $K$.  The same strategy applies with the geometric phase formulation calculating the winding of the Evans function in the winding in the fibers of the Hopf bundle.  This section develops the method of geometric phase with $\mathcal{L}$, the linearization of the complex Ginzburg-Landau equation about the Hocking-Stewartson pulse.

\begin{hyp}\label{hyp:K}  Assume that the contour $K$ is a smooth, simple closed curve in $\Omega \subset \mathbb{C}$ such that there is no spectrum of $\mathcal{L}$ in $K$.  Let $K^\circ$ be the region enclosed by $K$---assume $K^\circ$ is homeomorphic to the disk $D\subset\mathbb{R}^2$ and that $K$ is parametrized by $\lambda(s):[0,1]\hookrightarrow K$ with standard orientation.
\end{hyp}

The Hopf bundle $S^{11}$ is realized within the phase space of the exterior system on $\mathbb{C}^4$, $\Lambda^2(\mathbb{C}^4)\equiv \mathbb{C}^6$, by spherical projection; in general the generic Hopf bundle $S^{2n-1}$ is realized similarly in $\mathbb{C}^n$.  

\begin{mydef}
The Hopf bundle is a principal fiber bundle with full space $S^{2n-1}$, base space $\mathbb{C}P^{n-1}$, and fiber $S^1$. The fiber $S^1$ acts naturally on $S^{2n-1}$ by scalar multiplication by the associated element in the unitary group $U(1)$, and with respect to this action, the quotient is $\mathbb{C}P^{n-1}$.  The spaces are related by the diagram
\begin{equation}\begin{CD}
S^1 @>>> S^{2n-1}\\
& & @VVV \pi \\
& & \mathbb{C}P^{n-1}
\end{CD}\end{equation}  
where $\pi$ is the quotient map induced by the group action.
\end{mydef}
With respect to this coordinate realization of the Hopf bundle in the phase space $\mathbb{C}^n$, one may define a natural choice of connection.  A connection describes the motion of a path within the fibers of the bundle---it is possible to define a connection via a connection 1-form, as below.  For a full discussion on the theory of principal fiber bundles and connections the reader is referred to Kobayashi \& Nomizu \cite{kobayashi1996}.  

\begin{mydef}
For the Hopf bundle $S^{2n-1}$, viewed in coordinates in $\mathbb{C}^n$, define the connection 1-form $\omega$ pointwise for $p\in S^{2n-1}$ as a mapping of the tangent space of the Hopf bundle $T_p(S^{2n-1})\subset T_p(\mathbb{C}^n)$
\begin{equation} \begin{matrix}
\omega_p : & T_p(S^{2n-1}) & \rightarrow & i\mathbb{R} \cong \mathcal{G} \\ \\
 & V_p & \mapsto & \langle V_p, p \rangle_{\mathbb{C}^n} \\
\end{matrix} \end{equation}
where $\mathcal{G}$ is the Lie algebra of the fiber $S^1$ \cite{WAY2009}.  Defining a connection 1-form defines a connection and $\omega$ will be denoted the \textbf{natural connection} on the Hopf bundle.
\end{mydef}

\begin{remark}
The natural connection is a connection of the generic Hopf bundle $S^{2n-1}$ and it is the unique connection for $S^3$.
This is proven by Way \cite{WAY2009} in \S 3.5 and the reader is referred there for a full discussion.
\end{remark}
The realization of the Hopf bundle, and its connection, in coordinates for $\mathbb{C}^n$ also yields an alternative formulation for a non-zero path in $\mathbb{C}^n$, implicitly describing the connection of its spherical projection.  
\begin{lemma}
Let $u(t):[0,1]\rightarrow \mathbb{C}^n$ be a non-zero path, and $\hat{u}(t)$ be its the spherical projection into $S^{2n-1}$. The natural connection of the tangent vector of $\hat{u}(t)$ is identically
\begin{align}\label{eq:connection}
\omega\Big(\frac{d}{dt} \hat{u}(t)\Big)&=  i \frac{Im\big(\langle u'(t), u(t) \rangle\big)}{(\langle u(t),u(t) \rangle)}
\end{align}
\end{lemma}
\begin{proof}
By definition
\begin{align*}
\omega\left(\frac{d}{dt} \hat{u}(t)\right) &= \Big\langle \frac{d}{dt} \frac{u(t)}{\big\langle u(t),u(t)\big\rangle^\frac{1}{2}}, \frac{u(t)}{\big\langle u(t),u(t)\big\rangle^\frac{1}{2}} \Big\rangle\\
&= \Big\langle  \frac{u'(t)\langle u(t),u(t)\rangle^\frac{1}{2}}{\langle u(t),u(t)\rangle}  - \frac{ u(t) Re\big(\langle u'(t),u(t)\rangle\big)}{\langle u(t),u(t)\rangle^\frac{3}{2}}, \frac{u(t)}{\langle u(t),u(t)\rangle^\frac{1}{2}} \Big\rangle\\
&=   \frac{\langle u'(t),u(t)\rangle}{\langle u(t), u(t)\rangle} - \frac{ Re\big(\langle u'(t),u(t)\rangle\big)}{\langle u(t),u(t)\rangle}\\
&= i\frac{ Im \big( \langle u'(t), u(t) \rangle\big)}{\langle u(t),u(t)\rangle}  
\end{align*}
which verifies equation (\ref{eq:connection}).
\end{proof}
With the choice of connection, one may measure the parallel translation of a path in the Hopf bundle---the total parallel translation accumulated by a path is referred to as the geometric phase.
\begin{mydef}
Let $\gamma\subset S^{2n-1}$ be a path in the Hopf bundle.  Then the \textbf{geometric phase} of $\gamma$ with respect to the natural connection is the total winding in the fiber given as
\begin{align}
\label{eq:geometricphase}
GP(\gamma) &=\frac{1}{2\pi i} \int_{\gamma} \omega
\end{align}
For a general non-zero path $\gamma$ parametrized by $u(t):[0,1]\rightarrow \mathbb{C}^n$ one may write the geometric phase as
\begin{align}\label{eq:computephase}
GP(\gamma) &=\frac{1}{2\pi} \int^1_0 \frac{Im\big(\langle u'(s), u(s) \rangle\big)}{(\langle u(s),u(s) \rangle)}
\end{align}
\end{mydef}

In his thesis \cite{WAY2009}, Way developed numerical results which calculated the geometric phase of special solutions for the system (\ref{eq:A^2system}).  In particular, he considered solutions which corresponded to the stable manifold of the system $A_{+\infty}(\lambda)$ and, with respect to the equation (\ref{eq:geometricphase}), he calculated their phase with respect to the contour $K$.  The reformulation of the method of geometric phase, as in the proof by Grudzien, Bridges \& Jones \cite{Grudz2015}, is elaborated below:  
\begin{figure}[H]
\center
\textbf{The Method of Geometric Phase}
\begin{itemize}
\item \textbf{Step 1:} Choose a contour $K$ in $\mathbb{C}$ that \textbf{does not} intersect the spectrum of the operator $\mathcal{L}$, and let $\lambda(s):[0,1]\rightarrow K$ be a parametrization of $K$. 
\item \textbf{Step 2:} Varying $\lambda \in K$ define a loop of eigenvectors, $X^\pm(\lambda)$, for the $A^{(2)}_{\pm\infty}(\lambda)$ system (\ref{eq:A^2infty}) where $X^\pm(\lambda)$ corresponds to the dominant eigenvalue of positive real part. 
\item \textbf{Step 3:} Choose $x_0$ ``close'' to $-\infty$ and let $X^-(\lambda,x_0)$ be a loop of initial conditions.
\item \textbf{Step 4:} Integrate the loop of eigenvectors, following $Y'=A^{(2)}(\lambda,x)Y$, to $x_1$ ``close'' to $+\infty$
\item \textbf{Step 5:} Measure the \textbf{relative geometric phase} of $X^-(\lambda,\xi_1)$ and $X^+(\lambda)$, ie: 
\begin{align}\label{eq:Xgeometricphase}
GP\Big(X^-\big(K,x_1\big)\Big) - GP\Big(X^+\big(K\big)\Big)
\end{align}
where $GP(\gamma)$ is defined as in equation (\ref{eq:computephase}). 
\end{itemize}
\end{figure}
\begin{theorem}
Let $X^+(\lambda,\tau)$ be a solution to the system (\ref{eq:A^2system}) which is in the unstable manifold for the asymptotic system $A^{(2)}_{-\infty}(\lambda)$. As $ x_1 \rightarrow \infty$, the relative geometric phase given by (\ref{eq:Xgeometricphase}) converges to the multiplicity of the eigenvalues enclosed by the contour $K$.
\end{theorem}
\begin{proof}
This theorem is proven in a general formulation by Grudzien, Bridges \& Jones \cite{Grudz2015} and the reader is referred there for a discussion of the general method.
\end{proof}
\begin{remark}
In the above case where $A_{-\infty}(\lambda) \equiv A_{\infty}(\lambda) \equiv A_{+\infty}(\lambda)$, 
one may take $X^-(\lambda) \equiv X^+(\lambda)$ so that the formulation in equation (\ref{eq:Xgeometricphase}) is the difference of the geometric phase of the evolved solution and its initial condition---numerically this is given by
\begin{align}
GP\big(X^-(\lambda,x_1)\big) - GP\big(X^-(\lambda,x_0)\big)
\end{align}
\end{remark}

Note that the relative geometric phase need not be an integer, and indeed cannot always be an integer value, as can be seen directly in the proof of the method of geometric phase in \cite{Grudz2015}.  The phase calculation must be continuous in $x\in (-\infty,\infty)$ by the properties of the flow.  Moreover, when $X^-(\lambda) = X^+(\lambda)$ the relative geometric phase will always start at zero and transition to the multiplicity of the eigenvalue, as demonstrated in the following examples.

\section{Numerical Results}
\label{section:numerics}
This section describes the implementation of the numerical method and demonstrates the transition of the geometric phase.  In many Evans function formulations, solutions corresponding to the un/stable manifolds for the systems $A_{\mp\infty}(\lambda)$ are integrated to some matching value, usually $x=0$ where the winding of the Evans function is calculated \cite{Jones1984}.  When computing the winding of the Evans function with the geometric phase in the Hopf bundle, one considers the unstable \textit{or} stable manifold and computes the geometric phase accumulated relative to the opposite asymptotic condition.  The accumulation of the geometric phase thus introduces a new dependence on $x_1$ where $x_1$ is the final $x$ value of the forward integration described in the method of the geometric phase above.  As in the examples shown by Grudzien, Bridges \& Jones \cite{Grudz2015}, the geometric phase undergoes a transition as $x_1$ grows large, but for the computation of eigenvalues of the linearization of the complex Ginzburg-Landau equation about the Hocking-Stewartson pulse below, the transition isn't uniform across the eigenvalues or the integration parameter.  

In each example below, the contour $K$ is chosen to be the circle of radius $.1$ about $\lambda_0$ where $\lambda_0\in\{0,15,-6.6537\}$.  The contour is discretized into $10,000$ even steps, and for each fixed $\lambda$ in the discretization of $K$, the the unstable eigenvector (\ref{eq:ueigenvector}) is integrated from $x_0=-10$ forward to some $x_1$.  The Matlab ODE45 solver is used to find the trajectory of the initial condition $X^+(\lambda,x_0)$ with respect to the system (\ref{eq:A^2system}), and the trajectory is stored at step sizes of $.04$ in $x$.  To compute the relative phase in equation (\ref{eq:Xgeometricphase}), the derivative of a path $\frac{d}{ds}V(\lambda(s))$ is approximated via the difference equation  
\begin{align*}
\frac{V\left(\lambda(s)+\delta s\right) - V\left(\lambda(s)-\delta s\right)}{2\delta s}.
\end{align*}
From the connection equation (\ref{eq:connection}), the geometric phase of $X^+(\lambda)$ and $X^-(\lambda,x_1)$ is computed with the Euler method. The difference of phases in equation (\ref{eq:Xgeometricphase}) is computed for each stored value of $x$ and the plot of the accumulated geometric phase $X^-(\lambda,x_1)$ is given for each of the three contours below---because the system is symmetric, the \textbf{relative} geometric phase is describe by subtracting the \textbf{initial} geometric phase from the \textbf{terminal} geometric phase.  However, in each of these examples the initial geometric phase is approximately zero, so that the terminal value in the phase plot is approximately eigenvalue multiplicity once the transition has terminated.  

The first figure demonstrates the phase transition for the simple eigenvalue at $\lambda_0 = 15$.  This figure describes a relatively uniform phase transition, exhibited for other examples demonstrated by Grudzien, Bridges \& Jones \cite{Grudz2015}.

\begin{figure}[H]
\center
\includegraphics[width=\linewidth]{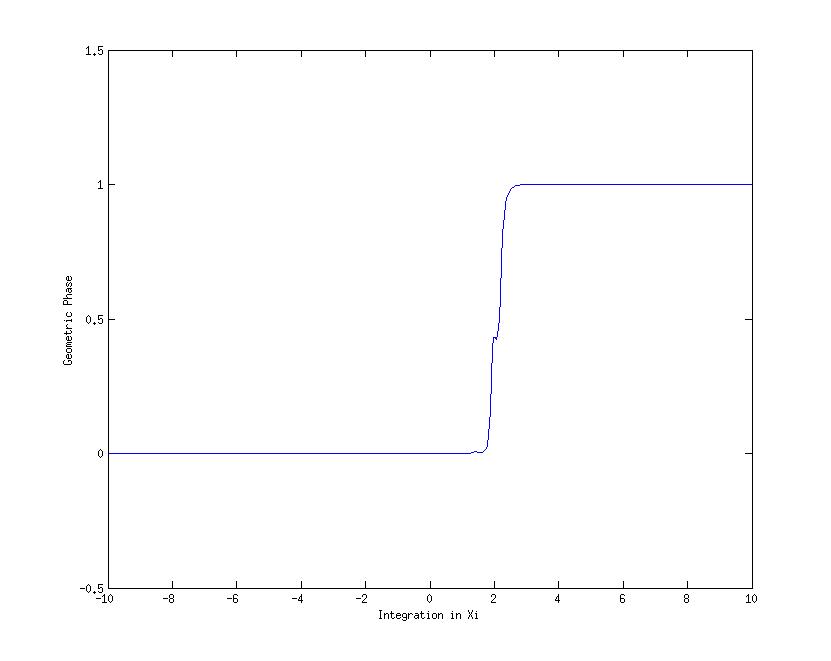}
\caption{The phase transition plotted for the simple eigenvalue at $\lambda\approx 15$}
\end{figure}

However, the other two plots for $\lambda_0 \in \{0,-6.6357\}$ demonstrate a non-uniform transition both in terms of the fluctuation in the phase calculation, as well as the value of $x$ for which the transition begins.
\begin{figure}[H]
\center
\includegraphics[width=\linewidth]{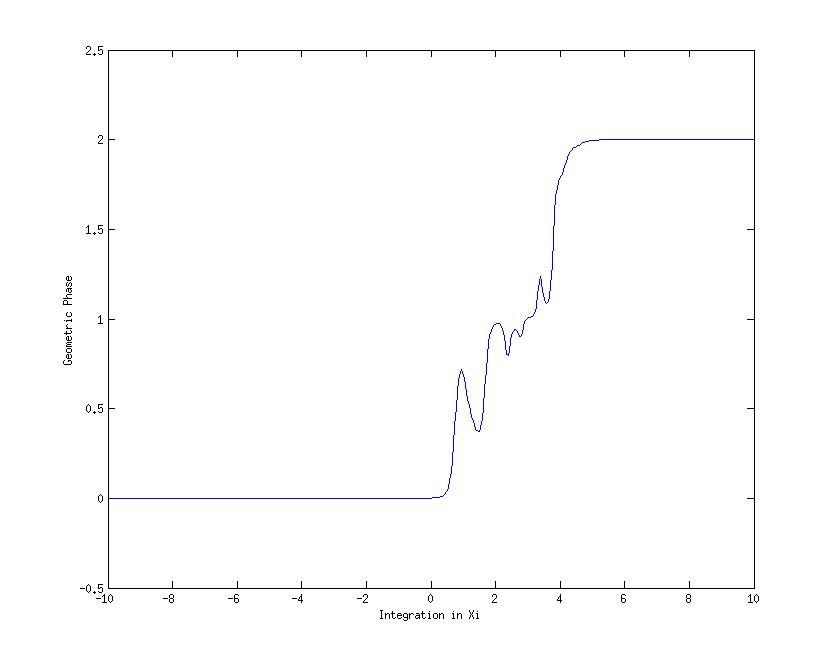}
\caption{The phase transition plotted for the double eigenvalue at $\lambda =0$.}
\end{figure}

\begin{figure}[H]
\center
\includegraphics[width=\linewidth]{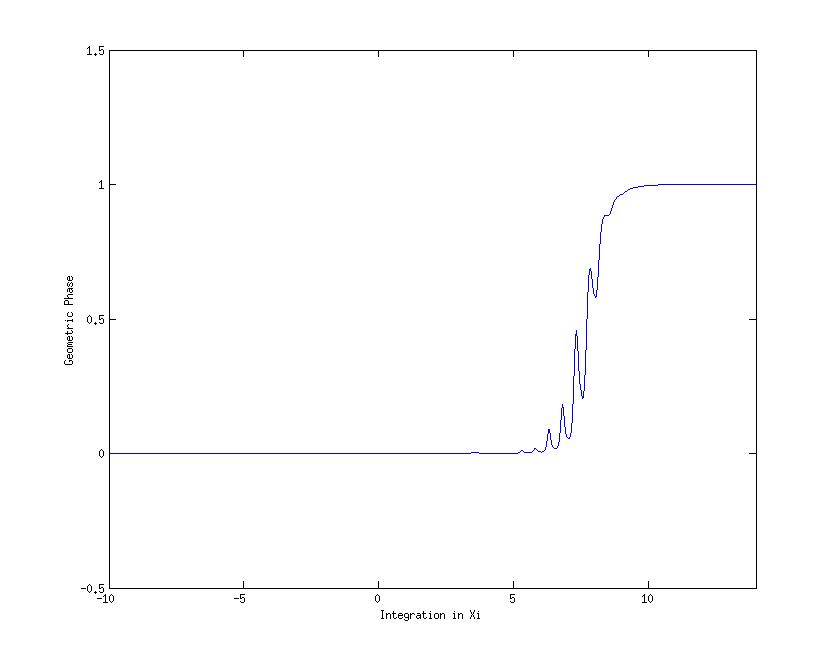}
\caption{The phase transition plotted for the simple eigenvalue at $\lambda\approx -6.6357$}
\end{figure}

The scale in $x$ direction for the plot of the phase transition at $\lambda_0 = -6.6357$ is longer, ending at $x_1 = 14$.  Noticeably, the transition here begins later, and doesn't terminate until it is nearly at the end of the other plots, at $x_1\approx 10$.  This example in particular highlights the importance of understanding the phase transition for applications.
\section{Discussion}
Evans function calculations are often useful as a stability index \cite{AGJ1990}, describing the multiplicity of eigenvalues of positive real part by computing the winding of the Evans function along the imaginary axis, and bounding the integral of the winding along a semi-circle of radius $r$, as $r\rightarrow\infty$.  In particular, in order to utilize the method of calculating the winding with the geometric phase in the Hopf bundle, it will be critical to understand the nature of the phase transition.  As demonstrated in the example above, the phase transition is neither uniform in the integration of the $x$ direction, nor uniform across eigenvalues---indeed the calculation may fluctuate and the initiation and termination of the transition differs for each of the above results.  For utilization as a stability index, one must understand the relationship between the transition and the underlying steady state to efficiently compute the eigenvalues.  The geometric phase must eventually converge to the multiplicity of the eigenvalues enclosed by the contour, but a theoretical understanding of the transition of the phase will be an important development for both the numerical method and the understanding the eigenvalue problem itself---indeed the method of the geometric phase offers a unique insight into the continuous accumulation of the eigenvalue as driven by the system dynamics, a new insight not afforded by other Evans function methods.

Currently the method of geometric phase is limited by the dependence on the exterior algebra formulation---for usual systems on $\mathbb{C}^n$, where the stable and unstable manifolds are of dimension approximately $\frac{n}{2}$, the dimension of phase space for the exterior algebra grows approximately exponentially in $n$, as discussed by Humpherys \& Zumbrun \cite{Hump2008}. 
 However, the fact that the method of geometric phase relies only on \textit{either} the unstable or stable manifold for the eigenvalue calculation highlights the potential for future reductions.
 Reducing the calculation of the phase to a \textit{frame of solutions} spanning the un/stable manifold, rather than the wedge product, is the subject of active research.
\section{Conclusion}
Way developed numerical results for the calculation of the geometric phase for solutions in the stable manifold of the $A^{(2)}_{-\infty}(\lambda)$ system for the linearization of the complex Ginzburg-Landau equation about the Hocking-Stewartson pulse, in his thesis \cite{WAY2009}.  This work builds on that discussion, framing the numerical method in terms of the relative phase as in the proof of the method of geometric phase by Grudzien, Bridges \& Jones \cite{Grudz2015}, as well as demonstrating the phase transition in the calculation with new numerical results.  In studying the numerical method, understanding the phase transition and its relationship to the underlying wave will be critically important for finite approximations and particularly for use of the geometric phase for stability indices.  The varied nature of the phase transition across eigenvalues for a single example highlights the need to understand this transition, and this work opens new theoretical questions for the advancement of Evans function techniques, and the method of geometric phase. 
\section{Acknowledgements}
This work benefited from the support of NSF SAVI award DMS-0940363, MURI award A100752 and GCIS award DMS-1312906.
\section*{References}
\bibliography{biblio}
\end{document}